\documentclass[12pt, reqno]{amsart}
\usepackage{}
\usepackage{mathrsfs}
\usepackage{amsfonts}
\usepackage{xcolor}
\usepackage{amsmath}
\usepackage{amsthm}
\usepackage{amssymb,enumerate}
\usepackage{cases}
\usepackage{bm}
\usepackage{cite}
\usepackage[pdfstartview=FitH, colorlinks=true,  linkcolor=red, citecolor=blue]{hyperref}
\numberwithin{equation}{section}
\newtheorem{thm}{Theorem}[section]
\newtheorem{lem}{Lemma}[section]
\newtheorem{cor}{Corollary}[section]
\newtheorem{rem}{Remark}[section]

\newtheorem{defn}{Definition}[section]

\newtheorem{con}{Condition}[section]
\newcommand{\nn}{\nonumber}
\newenvironment{proof of theorem 1.1}{{\it Proof of Theorem 1.1}.}{{\hfill $\square$%
    \hskip - \parfillskip}}
\begin{document}
\title [Asymptotic Convergence of Fully Nonlinear Curvature Flows]{Asymptotic Convergence for a Class of Fully Nonlinear Contracting Curvature Flows* }
\author{Yusha Lv }

\address{\parbox[l]{1\textwidth}{School of Mathematics and Statistics,
 Qilu University of Technology (Shandong Academy of Sciences), Ji'nan 250353,
China
}}
\email{yslv@qlu.edu.cn}

\author{Hejun Wang**}
\address{\parbox[l]{\textwidth}{School of Mathematics and Statistics,
Shandong Normal University, Ji'nan 250014, China}}
\email{wanghjmath@sdnu.edu.cn}


\thanks{*Supported in part by
Natural Science Foundation of Shandong (No.ZR2020QA003, No.ZR2020QA004) and China Postdoctoral Science Foundation (No.2020M682222)}
\thanks{**The corresponding author}

\maketitle

\begin{abstract}
In this paper, we study a class of fully nonlinear contracting curvature flows of closed, uniformly convex hypersurfaces in the Euclidean space $\mathbb R^{n+1}$ with the normal speed $\Phi$
given by $r^\alpha F^\beta$ or $u^\alpha F^\beta$, where $F$ is a monotone, symmetric, inverse-concave, homogeneous of degree one function of the principal curvatures, $r$ is the distance from the hypersurface to the origin and $u$ is the support function of hypersurface.
If  $\alpha\geq \beta+1$ when $\Phi=r^\alpha F^\beta$ or $\alpha> \beta+1$ when $\Phi=u^\alpha F^\beta$, we prove that the flow exists for all times and converges to the origin.  After proper rescaling, we prove that the normalized flow converges exponentially in the $C^\infty$ topology to a sphere centered at the origin.  Furthermore, for special inverse concave curvature function $F=K^{\frac{s}{n}}F_1^{1-s}(s\in(0, 1])$, where $K$ is Gauss curvature and $F_1$ is inverse-concave, we obtain the asymptotic convergence for the flow with $\Phi=u^\alpha F^\beta$ when $\alpha=\beta+1$. If $\alpha<\beta+1$, a counterexample is given for the above convergence when speed equals to $r^\alpha F^\beta$.
\end{abstract}

\section{Introduction}\label{intro}
Let's begin by reviewing the known results on the behavior of
contracting curvature flows. For the classical mean curvature flow,  Huisken \cite{H}  proved that convex hypersurfaces contract to a point in finite time,
becoming spherical in shape as the limit is approached. Later, Huisken \cite{H1} extended the contracting mean curvature flow to general Riemannian manifolds with suitable assumptions on curvatures. These results were established for a class of curvature flows with speed given by curvature functions with degree of homogeneity equals to one in principal curvatures, see \cite{A,A2,Chow,Chow1,G1}.

However, for contracting curvature flows with degree of homogeneity greater than one in principal curvatures, the behavior is more difficult to establish.   Under the assumption that initial hypersurface is
sufficiently pinched in the sense that $h_{ij}\geq C(\beta)Hg_{ij}$, Chow \cite{Chow} proved that flows by $K^\beta$ with $\beta\geq 1/n$  shrink convex hypersurfaces to round point in finite time.
Similar results were obtained for powers of the scalar curvature \cite{Al,AS}, and for $K^\beta$ with $\beta> 1/(n+2)$ without pinching condition \cite{BCD}.
Recently, with different pinching condition, Schulze \cite{Sc1,Sc2}, Guo, Li and Wu \cite{GLW,GLW1} proved that the convergence of flows with speeds equal to a power of the mean curvature were still hold in Euclidean space and in hyperbolic space, respectively.

Recently, several papers have considered the curvature flows with speed given by general curvature functions with degree of homogeneity greater than one.
For flows with speed equals to a power of curvature function $F$, where $F$ is inverse concavity and its dual function approaches zero on the boundary of the positive cone, Andrews, McCoy and Zheng \cite{AMZ} proved that flows contract to a point in finite time. Without the assumption of inverse concavity on $F$, the convergence of flows, for which speeds are homogeneous functions of degree greater than one,  was established by Andrews and McCoy \cite{AM}, if the initial hypersurface is pinched in the sense that $||\AA||^2\leq \delta H^2$.
More recently, Li and the first author in \cite{LL,LL1} proved contracting curvature flows in space forms, for which speed is given by a power of curvature function $F$, produce a spherical shape with a different pinching condition. For constrained
flows, the convergence to a round sphere was established in \cite{ACW,AW,BP,BS1,CS,GLW2}.

Curvature flow with speed depending not only on the curvatures has been considered. Bryan-Ivaki \cite{BI} and Ivaki \cite{Iva} studied the flows with speed given by curvatures and support function. For curvature flows with speed depending on curvatures and radial function is considered in \cite{LSW1} and \cite{LSW}. Recently, for expanding curvature flows, Ding and Li \cite{DL,DL1} considered the flow with speed given by $u^\alpha F^{-\beta}$, where $F$ is a monotone, symmetric, homogeneous of degree one function of the principal curvatures. When $F=\sigma_k$, the flow has been studied by Sheng and Yi in \cite{SY1}. For contracting curvature flows, Li, Sheng and Wang \cite{LSW1} studied the flows with speed given by $fr^\alpha K$, and they proved that, when $\alpha\geq n+1$, the normalised flow converges exponentially to a sphere centered at the origin in the $C^\infty$-topology. In \cite{LSW}, Li, Sheng and Wang considered a class of fully nonlinear curvature flows with speed equals to $r^\alpha E_k$, where $E_k$ is the $k$-th mean curvature, and obtained similar results  when $\alpha\geq k+1$. For an anisotropic contracting curvature flow with speed given by $fu^\alpha K^\beta$, Sheng and Yi \cite{SY} proved the existence and  convergence after appropriate normalisation. For the flow with speed equals to $r^\alpha E_k^\beta$, Li, Xu and Zhang \cite{LXZ} studied the convergence of $k$-convex and star-shaped hypersurfaces. This kind of curvature flows can be used to deal with the Minkowski problems \cite{BI,CHZ,Iva,LSW1,LL20,SY}.

Motivated by above results, in this paper, we  consider the following problem: Let $X_0: M^n\rightarrow \mathbb R^{n+1}$ be a smooth immersion of an $n$-dimensional closed convex hypersurface. We consider a one-parameter family of smooth immersions $X(M, t): M\times [0, T)\rightarrow \mathbb R^{n+1}$ which satisfy
\begin{align}\label{1-1}
    \begin{cases}\frac{\partial}{\partial t}X(x,t)=-r^\alpha F^\beta(\mathscr{W}(x,t))\nu(x,t),  \\X(\cdot,0)=X_0(\cdot),\end{cases}
\end{align}
or
\begin{align}\label{1-1b}\tag{1.1$'$}
    \begin{cases}\frac{\partial}{\partial t}X(x,t)=-u^\alpha F^\beta(\mathscr{W}(x,t))\nu(x,t),  \\X(\cdot,0)=X_0(\cdot), \end{cases}
\end{align}
where $\beta\geq 1$, $\nu(x,t)$ denotes the outer unit normal to the evolving hypersurface $M_t=X(M, t)$ at the point $X(x,t)$, $\mathcal W$ is the matrix of the Weingarten map, $r=|X(x, t)|$ is the distance from the point  $X(x, t)$ to the origin which called the radial function of $M_t$, $u$ is the support function of  $M_t$, and the function $F(\mathcal W)$ satisfies the following conditions:
\begin{con}\label{con-1}
\begin{enumerate}[(i)]
\item $ F(\mathcal W)=f(\lambda(\mathcal W))$, where $\lambda(\mathcal W)$ gives the eigenvalues of $\mathcal W$ and $f$ is a smooth, symmetric function
defined on $\Gamma_+=\{\lambda=(\lambda_1,\cdots,\lambda_n)\in \mathbb R^n: \lambda_i>0, i=1, \cdots, n\}$;

\item $f$ is strictly increasing in each argument: $\frac{\partial f}{\partial\lambda_i}>0$ on $\Gamma$, $\forall~i=1, \cdots, n$;

\item $f$ is homogeneous of degree one: $f(k\lambda)=k f(\lambda)$ for any $k>0$;

\item $f$ is strictly positive on $\Gamma_+$ and normalized to have $f(1, \cdots, 1)=1$.
\item $f$ is inverse concave, that is, the function
\begin{align*}
  f_*(\lambda_1, \cdots, \lambda_n)=f(\lambda_1^{-1},\cdots,\lambda_n^{-1})^{-1}
\end{align*}
is concave.

\item $f_*$ approaches zero on the boundary of $\Gamma_+$.
\end{enumerate}
\end{con}

\begin{rem}
There are many examples of curvature functions satisfying conditions (v) and (vi), for example, $F=E_k^{1/k}, k=1,\cdots,n$, the power
means  $F=(\frac{1}{n}\mathop{\sum}\limits_{i}\lambda_i^r)^{\frac{1}{r}}, r>0$ and convex function $F$. More examples can be constructed as
follows: If curvature functions $G_1$ and $G_2$ satisfy Condition \ref{con-1}, then $F=G^{s}_1 G^{1-s}_2$ satisfies
Condition \ref{con-1} for any $s\in[0, 1]$. (see \cite{A4,ALM} for more examples).
\end{rem}

We have the following asymptotic behaviour.
\begin{thm}\label{thm1}
Let $M_0$ be a smooth, closed and uniformly convex hypersurface in $\mathbb R^{n+1}$
enclosing the origin. If

(i) $\alpha\geq \beta+1$, then the flow \eqref{1-1}  has a unique smooth and uniformly
convex solution $M_t$ for all time $t>0$, which converges to the origin.

(ii) $\alpha>\beta+1$, then the flow \eqref{1-1b} has a unique smooth and uniformly
convex solution $M_t$ for all time $t>0$, which converges to the origin.

Furthermore, the rescaled hypersurface $\tilde M_t=\varphi(t)M_t$ converges exponentially in the $C^\infty$ topology to a sphere centered at the origin, where
\begin{align*}
\varphi(t)=\begin{cases}e^t, \quad\quad\quad\quad\quad\quad\quad\quad\quad\quad\quad\quad \alpha=\beta+1,
\\ \big(1+(\alpha-\beta-1)t\big)^{\frac{1}{\alpha-(\beta+1)}}, \quad\quad\alpha\neq \beta+1.\end{cases}
\end{align*}
\end{thm}

\begin{rem}
 The case $F=E_k^{\frac{1}{k}}$ and $\beta=k$ of \eqref{1-1} in Theorem \ref{thm1} reduces to Theorem 1.1 in \cite{LSW}.
\end{rem}

In particular, for special inverse concave curvature function $F=K^{\frac{s}{n}}F_1^{1-s}(s\in(0, 1])$, where $K$ is the Gauss curvature and $F_1$ satisfies Condition \ref{con-1}, the convergence of flow \eqref{1-1b} also holds when $\alpha=\beta+1$.

\begin{thm}\label{thm4}
Suppose curvature function $F_1$ satisfies Condition \ref{con-1} and $F=K^{\frac{s}{n}}F_1^{1-s}$, where $K$ is the Gauss curvature and $s\in(0, 1]$. Let $M_0$ be a smooth, closed and uniformly convex hypersurface in $\mathbb R^{n+1}$
enclosing the origin. If $\alpha=\beta+1$, then the flow \eqref{1-1b} has a unique smooth uniformly
convex solution $M_t$ for all time $t>0$, which converges to the origin.
Furthermore, the rescaled hypersurface $\tilde M_t=e^tM_t$ converges exponentially in the $C^\infty$ topology to a sphere centered at the origin.
\end{thm}

Now we define the rescaled immersion by
\begin{align*}
\tilde X(\cdot, \tau)=\varphi(t)X(\cdot, t),
\end{align*}
where
\begin{align*}
\tau=\begin{cases}t,~\quad\quad\quad\quad\quad\quad\quad\quad\alpha=\beta+1,
\\ \frac{\log\big(1+(\alpha-\beta-1)t\big)}{\alpha-\beta-1}, \quad\quad\alpha\neq \beta+1.\end{cases}
\end{align*}
After a standard computation, $\tilde X(\cdot, \tau)$ satisfies the following normalized flow
\begin{align}\label{1-2}
 \begin{cases}\frac{\partial}{\partial t}X(x,t)=-r^\alpha F^\beta\nu(x,t)+X(x, t),
 \\X(\cdot,0)=X_0(\cdot),\end{cases}
\end{align}
or
\begin{align}\label{1-2b}\tag{1.2$'$}
 \begin{cases}\frac{\partial}{\partial t}X(x,t)=-u^\alpha F^\beta\nu(x,t)+X(x, t),
 \\X(\cdot,0)=X_0(\cdot),\end{cases}
\end{align}
where we used $t$ instead of $\tau$ to denote the time variable and omit the `tilde'.
The study of the asymptotic behaviour of the flow \eqref{1-1} or \eqref{1-1b} is equivalent to the long time behaviour of the normalized flow \eqref{1-2} or \eqref{1-2b}.

When $\alpha<\beta+1$, we find that the hypersurfaces evolving by \eqref{1-1} may reach the origin in finite time, before the hypersurface shrinks to a point. Therefore the smooth convergence to a round point does not
hold in general.

\begin{thm}\label{thm3}
Suppose $\alpha<\beta+1$. There exists a smooth closed uniformly convex hypersurface $M_0$ such that under the flow \eqref{1-1}, for some $T>0$,
\begin{align*}
\lim_{t\rightarrow T}\mathcal R(X(x, t))=\lim_{t\rightarrow T}\frac{\max_{\mathbb S^n} r(\cdot, t)}{\min_{\mathbb S^n} r(\cdot, t)}=\infty.
\end{align*}
\end{thm}

The rest of the paper is organized as follows. First we recall some notations, known
results and evolution equations for some geometric quantities in Section 2. In Section 3, we established a priori estimates for higher order derivative of radial function, which ensure the longtime existence of the normalized flow \eqref{1-2} or \eqref{1-2b}. Section 4 is devoted to the proof of exponential convergence of normalized flow to a sphere centered at the origin in the $C^\infty$-topology. Finally in Section 5, we give the proof of Theorem
\ref{thm3}.

\vskip 0.5cm
\section{Notations and Preliminary Results}\label{Preliminaries}


Suppose $M$ is a hypersurface of $\mathbb{R}^{n+1}$,  let $A=\{h_{ij}\}$ and $\mathcal W=\{g^{ik}h_{kj}\}=\{h^i_j\}$ denote the second fundamental form and the Weingarten map respectively. The eigenvalues $\lambda_i$ with $\lambda_{\min}=\lambda_1\leq \cdots\leq \lambda_n=\lambda_{\max}$ of the second fundamental form  $A$ with respect to the metric $\{g_{ij}\}$
are called the principal curvatures of $X(M)$. Let the mean
curvature and the Gauss curvature are denoted by $H$ and $K$  respectively.


%
%

Let $M_t$ be smooth strictly convex hypersurface, then it can be parameterized as a graph  over the unit sphere, that is
$$X(x, t)=r(\xi(x, t), t)\xi(x, t),$$
where $\xi(\cdot, t):\mathbb S^n\rightarrow \mathbb S^n$ and $r(\cdot, t):\mathbb S^n\rightarrow \mathbb R_+$ is the radial function of $M_t$.  Let $e_1,\cdots, e_n $
be a smooth local orthonormal frame field on $\mathbb S^n$, and let $D$ be the covariant derivative on $\mathbb S^n$.
Then the induced metric $g_{ij}$, the inverse of the metric $g^{ij}$, the unit normal $\nu$, and the second fundamental form $h_{ij}$ can be
 written in terms of $r$ and whose spatial derivatives as follows:
\begin{align}\label{2-3}
 \begin{split}
g_{ij}&=r^2\delta_{ij}+r_ir_j,\\
g^{ij}&=r^{-2}\left(\delta_{ij}-\frac{r_ir_j}{r^2+|Dr|^2}\right),\\
\nu&=\frac{rz-Dr}{\sqrt{r^2+|Dr|^2}},\\
h_{ij}&=\frac{1}{\sqrt{r^2+|Dr|^2}}(-rr_{ij}+2r_ir_j+r^2\delta_{ij}),
 \end{split}
\end{align}
where $r_i=D_ir$ and $r_{ij}=D^2_{ij}r$.

By \eqref{2-3}, we can deduce that the normalized flow \eqref{1-2} or \eqref{1-2b} can be described by the following scalar equation for $r(\cdot, t)$
\begin{align}\label{2-4}
    \begin{cases}
    \frac{\partial r}{\partial t}(\xi, t)=-\sqrt{1+r^{-2}|Dr|^2}r^\alpha F^{\beta}(\xi,t)+r(\xi, t), \quad {\text on}~\mathbb S^n\times[0, \infty), \\r(\cdot,0)=r_0(\cdot),\end{cases}
\end{align}
or
\begin{equation}\label{2-4b}\tag{{2.2}$'$}
    \begin{cases}
    \frac{\partial r}{\partial t}(\xi, t)=-\sqrt{1+r^{-2}|Dr|^2}u^\alpha F^{\beta}(\xi,t)+r(\xi, t), \quad {\text on}~\mathbb S^n\times[0, \infty), \\r(\cdot,0)=r_0(\cdot),\end{cases}
\end{equation}
where $r_0$ is the support function of the initial hypersurface $M_0$.

If the hypersurface $M_t$ is uniformly convex,  then its support function is defined by
$$u(z, t)=\sup\{\langle x, z\rangle: x\in\Omega_t,~z\in \mathbb S^n\},$$
 where $\Omega_t$ is a convex body enclosed by $M_t$. Then hypersurface $M_t$ can be given by the embedding (ref. \cite{AMZ})
$$X(z, t)=u(z, t)z+ Du(z, t),$$
where $D$ is the gradient with respect to the standard  metric $\sigma_{ij}$ and connection on $\mathbb S^n$. The derivative of this map is given by
\begin{align*}
  \partial_iX=\tau_{ik}\sigma^{kl}\partial_lz,
\end{align*}
where $\tau_{ij}$ has the form
\begin{align}\label{2-6}
\tau_{ij}= D_i D_ju+\delta_{ij}u.
\end{align}
In particular the eigenvalues of $\tau_{ij}$ with respect to the
metric $\sigma_{ij}$ are the inverses of the principal curvatures, or the principal radii of curvatures.

Thus the solution of normalized flow \eqref{1-2} or \eqref{1-2b} is then given, up to a time dependent diffeomorphism, by solving the following scalar parabolic equation
\begin{align}\label{2-5}
    \begin{cases}
    \frac{\partial u}{\partial t}(x, t)=-r^\alpha F^{-\beta}_*(\tau_{ij})+u(x, t), \quad {\text on}~\mathbb S^n\times[0, \infty),
     \\u(\cdot,0)=u_0(\cdot),\end{cases}
\end{align}
or
\begin{align}\label{2-5b}\tag{2.4$'$}
    \begin{cases}
    \frac{\partial u}{\partial t}(x, t)=-u^\alpha F^{-\beta}_*(\tau_{ij})+u(x, t), \quad {\text on}~\mathbb S^n\times[0, \infty),
     \\u(\cdot,0)=u_0(\cdot),\end{cases}
\end{align}
where $u_0$ is the support function of the initial hypersurface $M_0$ and
$$r=\sqrt{u^2+|Du|^2}.$$

For a curvature function $F$ in Section 1, let $\dot{F}^{kl}$  denote the matrix of the first order partial derivatives with respect
to the components of its argument
$$\frac{d}{ds}F(A+sB)\Big|_{s=0}=\dot{F}^{kl}\Big|_A B_{kl}.$$
And the second order partial derivatives of $F$ are given by
$$\frac{d^2}{ds^2}F(A+sB)\Big|_{s=0}=\ddot{F}^{kl,rs}\Big|_A B_{kl}B_{rs}.$$
Similarly, the derivatives of $f$ with respect to $\lambda$ are denoted by
$$\dot f^i(\lambda)=\frac{\partial f}{\partial\lambda_i}(\lambda)\quad\text{ and}\quad\ddot f^{ij}(\lambda)=\frac{\partial^2 f}{\partial\lambda_i\partial\lambda_j}(\lambda).$$
In what follows, we will drop the arguments  when derivatives of $F$ and $f$ are evaluated at $\mathcal W$ or $\lambda(\mathcal W)$ respectively.
At any diagonal $A$ with distinct eigenvalues, the second derivative $\ddot F$ in direction $B\in{\rm Sym}(n)$ is given in terms of $\dot f$ and $\ddot f$ by (see \cite{A,A4}):
\begin{align}\label{2-1}
  \ddot F^{ij,kl}B_{ij}B_{kl}=\sum_{i,k}\ddot f^{ik}B_{ii}B_{kk}+2\sum_{i>k}\frac{\dot f^i-\dot f^k}{\lambda_i-\lambda_k}B_{ik}^2.
\end{align}

 The following evolution equations can be derived by computations as in Sect. 3 of \cite{H} (see also \cite{G2,LL,LL1,LSW}).
\begin{lem}\label{lem2-2}
Denote $\Phi=r^\alpha F^\beta$ or $u^\alpha F^\beta$, then under the normalized flow \eqref{1-2} or \eqref{1-2b}, we have the following evolution equations
\begin{align}
\partial_tg_{ij}&=-2\Phi h_{ij}+2g_{ij},\\
\partial_t \nu&=\nabla\Phi,
\end{align}
and
\begin{align}
\partial_th_{ij}&=\beta\Phi F^{-1}\dot F^{kl}h_{ij, kl}+\beta\Phi F^{-1}\ddot F^{kl, pq}h_{kl,i}h_{pq, j}+\beta\Phi F^{-1}\dot F^{kl}h_{kp}h^p_lh_{ij}\nonumber\\
&\quad+\beta\Phi F^{-1}\dot F^{kl}(h_{il}h_{kp}h^p_j-h_{ip}h^p_lh_{kj})-(\beta+1)\Phi h_{ip}h^p_j+h_{ij}\nn\\
&\quad+\beta(\beta-1)\Phi(\nabla_i\log F)(\nabla_j\log F)+2\alpha\beta\Phi(\nabla_i\log r)(\nabla_j\log F)\nn\\
&\quad+\alpha(\alpha-1)\Phi (\nabla_i\log r)(\nabla_j\log r)+\alpha\Phi\frac{1}{r}r_{ij},
\end{align}
or
\begin{align}
\partial_th_{ij}&=\beta\Phi F^{-1}\dot F^{kl}h_{ij, kl}+\beta\Phi F^{-1}\ddot F^{kl, pq}h_{kl,i}h_{pq, j}+\beta\Phi F^{-1}\dot F^{kl}h_{kp}h^p_lh_{ij}\nonumber\\
&\quad+\beta\Phi F^{-1}\dot F^{kl}(h_{il}h_{kp}h^p_j-h_{ip}h^p_lh_{kj})-(\beta+1)\Phi h_{ip}h^p_j+h_{ij}\nn\\
&\quad+\beta(\beta-1)\Phi(\nabla_i\log F)(\nabla_j\log F)+2\alpha\beta\Phi(\nabla_i\log u)(\nabla_j\log F)\nn\\
&\quad+\alpha(\alpha-1)\Phi (\nabla_i\log u)(\nabla_j\log u)+\alpha\Phi\frac{1}{u}u_{ij},\tag{2.8$'$}
\end{align}
when $\Phi=r^\alpha F^\beta$ or $u^\alpha F^\beta$, respectively.
\end{lem}

\section{A Priori Estimates}
In this section, we establish the a priori estimates and show that the normalized flow exists for all times. We first derive the $C^0$-estimate for the radial function and the support function.
\begin{lem}\label{lem3-1}
Let $X(\cdot, t)$ be a smooth, uniformly convex solution to \eqref{1-2} or \eqref{1-2b} for $t\in[0, T)$.
If $\alpha\geq\beta+1$, then there exists a positive constant $C_1$, depending only on $\mathop{\max}\limits_{\mathbb S^n}r(\cdot,0)$ and $\mathop{\min}\limits_{\mathbb S^n}r(\cdot,0)$,  such that the radial function $r(\cdot, t)$ and the support function $u(\cdot,t)$ of $X(\cdot,t)$ satisfy
\begin{align*}
\frac{1}{C_1}\leq r(\cdot, t)\leq C_1, \quad\quad t\in[0, T).
\end{align*}
and
\begin{align*}
\frac{1}{C_1}\leq u(\cdot, t)\leq C_1, \quad\quad t\in[0, T).
\end{align*}
\end{lem}
\begin{proof}
First we prove the bounds for radial function $r(\cdot, t)$. Let
$$ r_{\max}(t)=\mathop{\max}\limits_{\mathbb S^n}r(\cdot, t)\quad\text{ and}\quad
 r_{\min}(t)=\mathop{\min}\limits_{\mathbb S^n}r(\cdot, t).$$
By \eqref{2-3}, at the point where radial function $r$ attains its spatial maximum, we have the principal curvatures $\lambda_i~(i=1,\cdots, n)$ of hypersurface $M_t$ satisfy $$\lambda_i\geq \frac{1}{r_{\max}},\quad\quad i\in\{1, \cdots, n\},$$
that is
$$F\geq \frac{1}{r_{\max}}.$$
From
\begin{align}\label{3-2}
r=\sqrt{u^2+|Du|^2},
\end{align}
it follows that
\begin{align}\label{3-20}
\mathop{\max}\limits_{\mathbb S^n}r(\cdot, t)=\mathop{\max}\limits_{\mathbb S^n}u(\cdot, t)\quad\text{ and}\quad \mathop{\min}\limits_{\mathbb S^n}r(\cdot, t)=\mathop{\min}\limits_{\mathbb S^n}u(\cdot, t).
\end{align}
Thus, by \eqref{2-4} or \eqref{2-4b}, we have
\begin{align}\label{3-1}
\frac{d}{dt}r_{max}\leq r_{\max}(1-r_{\max}^{\alpha-\beta-1}).
\end{align}

When $\alpha=\beta+1$, by \eqref{3-1} we have $\frac{d}{dt}r_{max}\leq0$, therefore $r(\cdot, t)\leq \mathop{\max}\limits_{\mathbb S^n}r(\cdot, 0)$.

When $\alpha>\beta+1$, we can assume that $r_{\max}>1$, otherwise the conclusion holds. Hence \eqref{3-1} gives $\frac{d}{dt}r_{max}\leq0$, which implies
$$r(\cdot, t)\leq \max\{1,~\mathop{\max}\limits_{\mathbb S^n}r(\cdot, 0)\}.$$

From above arguments, we obtain the upper bound for $r(\cdot, t)$. The lower bound for $r(\cdot, t)$ follows similarly.

The bounds for support function $u$ can be obtained easily by equations \eqref{3-2}, \eqref{3-20} and the bounds for radial function $r$.
\end{proof}

The $C^1$-estimates for the radial function and the support function are the byproducts of $C^0$-estimate.
\begin{lem}\label{lem3-2}
Let $X(\cdot, t)$ be a smooth, uniformly convex solution to \eqref{1-2} or \eqref{1-2b} for $t\in[0, T)$. Then there exists a positive constant $C_2$, depending only on $\mathop{\max}\limits_{\mathbb S^n\times[0, T)}r$ and $\mathop{\min}\limits_{\mathbb S^n\times[0, T)}r$,  such that
\begin{align*}
|Dr(\cdot, t)|\leq C_2,
\end{align*}
and
\begin{align}\label{3-5}
|Du(\cdot, t)|\leq \mathop{\max}\limits_{\mathbb S^n\times[0, T)}r
\end{align}
for all $t<T$.
\end{lem}
\begin{proof}
Let $\xi=\frac{X(x, t)}{|X(x, t)|}$. Since $M_t$ are uniformly convex, $X(x, t)=r(\xi, t)\xi$. By \eqref{2-3}, we have the support function and radial function satisfy
\begin{align*}
u=\frac{r^2}{\sqrt{r^2+|Dr|^2}}.
\end{align*}
Hence,
$$|Dr(\xi, t)|\leq \frac{r^2}{u}\leq \frac{\max_{\mathbb S^n\times[0, T)}r^2}{\min_{\mathbb S^n\times[0, T)}r}.$$
Therefore, the upper bound for $|Dr|$ is obtained.

The estimate for $|Du|$ follows from equation \eqref{3-2}.
\end{proof}

Next, we establish the upper bound for the principal curvatures. For this purpose, the upper bound for $F$ is needed.
\begin{lem}\label{lem3-4}
Let $X(x, t)$ be a smooth, closed, uniformly convex solution to the normalized
flow \eqref{1-2} or \eqref{1-2b} for $t\in[0, T)$, which encloses the origin. Then there exists a positive constant $C_3$, depending only on $\alpha,\beta$, $\mathop{\max}\limits_{\mathbb S^n\times[0, T)}r$ and $\mathop{\min}\limits_{\mathbb S^n\times[0, T)}r$, such that
$$F(\cdot, t)\leq C_3, \quad\quad t\in[0, T).$$
\end{lem}
\begin{proof}
Let $\eta=\frac{1}{2}\mathop{\min}\limits_{\mathbb S^n\times[0, T)}u$, and consider the function
\begin{align*}
G(x, t)=\frac{-u_t}{u-\eta}=\frac{r^\alpha F_*^{-\beta}-u}{u-\eta}.
\end{align*}
Suppose $G(x_0, t_0)=\mathop{\min}\limits_{\mathbb S^n\times[0, T)}G(x, t).$ Then at point $(x_0, t_0)$, there holds
\begin{align}\label{3-6}
0=D_iG=\frac{-u_{ti}}{u-\eta}+\frac{u_tu_i}{(u-\eta)^2},
\end{align}
and
\begin{align}\label{3-7}
0\geq D^2_{ij}G&=\frac{-u_{tij}}{u-\eta}+\frac{2u_{ti}u_j+u_tu_{ij}}{(u-\eta)^2}
-\frac{2u_tu_iu_j}{(u-\eta)^3}\nn\\
&=\frac{-u_{tij}}{u-\eta}+\frac{u_tu_{ij}}{(u-\eta)^2}.
\end{align}
It follows from \eqref{2-6} and \eqref{3-7} that
\begin{align}\label{3-8}
-u_{tij}-u_t\delta_{ij}\leq Gu_{ij}-u_t\delta_{ij}=(\tau_{ij}-\eta\delta_{ij})G.
\end{align}

\textbf{Case I}: $\Phi=r^\alpha F^\beta$. By \eqref{3-2} and \eqref{3-6}, we have
\begin{align}\label{3-9}
\frac{\partial r}{\partial t}=\frac{uu_t+u_ku_{kt}}{r}=\frac{\eta u-r^2}{r}G.
\end{align}
Notice that  $\dot F^{ij}_*\tau_{ij}=F_*$, combining \eqref{2-5}, \eqref{3-8} with \eqref{3-9} gives
\begin{align}\label{3-10}
\partial_t G&=\frac{-u_{tt}}{u-\eta}+G^2=\frac{(r^\alpha F^{-\beta}_*(\tau_{ij})-u)_t}{u-\eta}+G^2\nn\\
&=\frac{\alpha r^{\alpha-1}F^{-\beta}_*r_t-u_t}{u-\eta}-\frac{\beta r^\alpha F^{-\beta-1}_*\dot F^{ij}_*\tau_{ijt}}{u-\eta}+G^2\nn\\
&=\frac{\alpha r^{\alpha-1}F^{-\beta}_*}{u-\eta}\cdot\frac{u\eta-r^2}{r}G+\frac{\beta r^\alpha F^{-\beta-1}_*\dot F^{ij}_*(-u_{ijt}-u_t\delta_{ij})}{u-\eta}+G+G^2\nn\\
&\leq \frac{\beta r^\alpha F^{-\beta-1}_*\dot F^{ij}_*(G\tau_{ij}-G\eta\delta_{ij})}{u-\eta}+C(G+G^2)\nn\\
&= \frac{\beta r^\alpha F^{-\beta}_*G}{u-\eta}(1-\eta\frac{1}{F_*}\dot F^{ij}_*\delta_{ij})+C(G+G^2)\nn\\
&\leq \frac{\beta r^\alpha F^{\beta}G}{u-\eta}(1-\eta F)+C(G+G^2),
\end{align}
where $\dot F^{ij}_*\delta_{ij}\geq 1$ (\cite{U}) was used in the last inequality.

\textbf{Case II}: $\Phi=u^\alpha F^\beta$.
By $\dot F^{ij}_*\tau_{ij}=F_*$,  \eqref{2-5b} and \eqref{3-8}, we also have
\begin{align}\label{3-10b}
\partial_t G&=\frac{-u_{tt}}{u-\eta}+G^2=\frac{(u^\alpha F^{-\beta}_*(\tau_{ij})-u)_t}{u-\eta}+G^2\nn\\
&=-\alpha u^{\alpha-1}F^{-\beta}_*G+\frac{\beta u^\alpha F^{-\beta-1}_*\dot F^{ij}_*(-u_{ijt}-u_t\delta_{ij})}{u-\eta}+G+G^2\nn\\
&\leq \frac{\beta u^\alpha F^{\beta}G}{u-\eta}(1-\eta F)+C(G+G^2).
\end{align}

Without loss of generality, we assume that $F^\beta\thickapprox G\gg 1$. Thus \eqref{3-10} or \eqref{3-10b} can be rewrite as
$$\partial_t G\leq \Lambda_1 G^2(\Lambda_2-\eta G^{\frac{1}{\beta}}),$$
where $\Lambda _1$ and $\Lambda_2$ only depending on $\alpha, \beta$, $\mathop{\max}\limits_{\mathbb S^n\times[0, T)}r$ and $\mathop{\min}\limits_{\mathbb S^n\times[0, T)}r$.
Therefore $G\leq C$. As a conclusion, the upper  bound for $F$ follows.
\end{proof}

As a byproduct of above lemma, the upper bounds for the principal curvatures are obtained.
\begin{cor}\label{cor-1}
Let $X(x, t)$ be a smooth, closed and uniformly convex solution to the normalized flow \eqref{1-2} or \eqref{1-2b} for $t\in[0, T)$, which encloses the origin.
Then there exists a positive constant $C$, depending only on $\alpha, \beta, M_0$, $\mathop{\max}\limits_{\mathbb S^n\times[0, T)}r$ and $\mathop{\min}\limits_{\mathbb S^n\times[0, T)}r$, such that the principal curvatures of $X(\cdot, t)$ are bounded from  above, i.e.,
$$\lambda_i(\cdot, t)\leq C$$
for all $t\in[0, T)$ and $i=1,\ldots, n.$
\end{cor}
\begin{proof}
By Lemma \ref{lem3-4}, we have the dual function satisfies $F_*\geq C$. Since $F_*$ approaches zero on the boundary of positive cone $\Gamma_+$, there exists $c>0$ such that $\frac{1}{\lambda_i}\geq c$, which implies $\lambda_i\leq C$.
\end{proof}

Next, we prove the lower bound for the principal curvatures.
Before we start, let us recall the following lemma, which describes hypersurfaces $M_t$ are ``uniformly star-shaped".
\begin{lem}[\cite{LSW}]\label{lem3-5}
Let $X(x, t)$ be a smooth, uniformly convex hypersurface which
solves the normalized flow \eqref{1-2}, and encloses the origin. Then for
any $t\in[0, T)$ and $p\in M_t$, and any unit tangential vector $e(p)\in T_pM_t$, we have
\begin{align*}
\left\langle e(p), \frac{p}{|p|}\right\rangle^2\leq 1-\delta_0,
\end{align*}
where $\delta_0>0$ is a small constant only depending on $\alpha, \beta$, $\mathop{\min}\limits_{\mathbb S^n\times[0, T)}r$ and $\mathop{\max}\limits_{\mathbb S^n\times[0, T)}r$.
\end{lem}

Besides, we also need the following algebra lemma.
\begin{lem}[\cite{LSW,U}]
Suppose $F$ satisfies Condition \ref{con-1} and $\{b^{ij}\}$ be the inverse matrix of $\{h_{ij}\}$. Then
\begin{align}\label{3-11}
(\ddot F^{ij, kl}+2\dot F^{ik}b^{jl})B_{ij}B_{kl}\geq 2 F^{-1}(\dot F^{ij}B_{ij})^2.
\end{align}
\end{lem}

\begin{lem}\label{lem3-6}
Let $X(x, t)$ be a smooth, closed and uniformly convex solution to the normalized flow \eqref{1-2} or \eqref{1-2b} for $t\in[0, T)$, which encloses the origin. Assume $\alpha\geq\beta+1$ when $\Phi=r^\alpha F^\beta$ or $\alpha>\beta+1$ when $\Phi=u^\alpha F^\beta$.
Then there exists a positive constant $C_4$, depending only on $\alpha, \beta$, $\mathop{\max}\limits_{\mathbb S^n\times[0, T)}r$ and $\mathop{\min}\limits_{\mathbb S^n\times[0, T)}r$, such that the principal curvatures of $X(\cdot, t)$ are bounded from below, i.e.,
$$\lambda_i(\cdot, t)\geq C_4,$$
for all $t\in[0, T)$ and $i=1,\cdots, n.$
\end{lem}
\begin{proof}
Let $\{b^{ij}\}$ be the inverse matrix of $\{h_{ij}\}$. We only need to prove the upper bound of principal radial of curvatures. For this purpose, let us define the following functions
$$\Lambda(x, t)=\max\{b^{ij}\xi_i\xi_j:g^{ij}(x, t)\xi_i\xi_j=1\},$$
and
$$W(x, t)=\log \Lambda(x, t)-\log u(x, t),$$
where $u(x, t)$ is the support function of $M_t$.

Fix $0<T'< T$, suppose $W$ attains its maximum on $\mathbb S^n\times[0, T']$ at $(x_0, t_0)$ with $t_0>0$, that is,
$$\max_{\mathbb S^n\times[0, T']}W=W(x_0, t_0), \quad\quad t_0>0.$$
Choose Riemannian normal coordinates at $(x_0, t_0)$ such that at this point we have
$$g_{ij}=\delta_{ij}, \quad h_{ij}=\lambda_i\delta_{ij}.$$
By a rotation, we may suppose that $\Lambda(x_0, t_0)=b^{ij}(x_0, t_0)\xi_i\xi_j$ with $\xi=(1, 0,\ldots, 0)$.

Let
$$w(x, t)=\log \rho(x, t)-\log u(x, t),$$
where $\rho(x, t)=\frac{b^{11}}{g^{11}}$. Then $w$ attains its maximum at point $(x_0, t_0)$.

\textbf{Case I}: $\Phi=r^\alpha F^\beta$. By the evolution equations for $g_{ij}$ and $h_{ij}$ in Lemma \ref{lem2-2}, we have, at point $(x_0, t_0)$,
\begin{align}\label{3-12}
\partial_t \rho&=-(b^{11})^2\partial_t h_{11}+b^{11}\partial_t g_{11}\nn\\
&=-(b^{11})^2\beta\Phi F^{-1}\dot F^{kl}h_{11, kl}-(b^{11})^2\beta\Phi F^{-1}\ddot F^{kl, pq}h_{kl,1}h_{pq, 1}-b^{11}\beta\Phi F^{-1}\dot F^{kl}h_{kp}h^p_l\nonumber\\
&\quad-(b^{11})^2\Phi\Big(\beta(\beta-1)(\nabla_1\log F)^2+2\alpha\beta(\nabla_1\log r)(\nabla_1\log F)\nn\\
&\quad\quad\quad\quad\quad\quad+\alpha(\alpha-1) (\nabla_1\log r)^2\Big)+(\beta-1)\Phi +b^{11}-\alpha\Phi(b^{11})^2\frac{1}{r}r_{11}.
\end{align}
From the definition of $\rho(x, t)=\frac{b^{11}}{g^{11}}$, it follows that, at point $(x_0, t_0)$,
\begin{align*}
\nabla^2_{ij}\rho&=(-b^{1p}b^{1q}h_{pq,i})_j\\
&=2b^{1r}b^{ps}b^{1q}h_{rsj}h_{pqi}-b^{1p}b^{1q}h_{pq,ij}\\
&=-(b^{11})^2h_{11, ij}+2(b^{11})^2b^{pq}h_{ip1}h_{jq1}.
\end{align*}
Thus, equation \eqref{3-12} can be rewrite as
\begin{align*}
\partial_t \rho
&=\beta\Phi F^{-1}\dot F^{kl}\nabla^2_{kl}\rho-(b^{11})^2\beta\Phi F^{-1}(\ddot F^{kl, pq}h_{kl,1}h_{pq, 1}+2\dot F^{kl}b^{pq}h_{kp1}h_{lq1})\\
&\quad-b^{11}\beta\Phi F^{-1}\dot F^{kl}h_{kp}h^p_l+(\beta-1)\Phi +b^{11}-\alpha\Phi(b^{11})^2\frac{1}{r}r_{11}\\
&\quad-(b^{11})^2\Phi\Big(\beta(\beta-1)(\nabla_1\log F)^2+2\alpha\beta(\nabla_1\log r)(\nabla_1\log F)\\
&\quad\quad\quad\quad\quad\quad+\alpha(\alpha-1) (\nabla_1\log r)^2\Big)\nn\\
&\leq \beta\Phi F^{-1}\dot F^{kl}\nabla^2_{kl}\rho+(\beta-1)\Phi +b^{11}-\alpha\Phi(b^{11})^2\frac{1}{r}r_{11}\nn\\
&\quad-(b^{11})^2\Phi\Big(\beta(\beta+1)(\nabla_1\log F)^2+2\alpha\beta(\nabla_1\log r)(\nabla_1\log F)\nn\\
&\quad\quad\quad\quad\quad\quad+\alpha(\alpha-1) (\nabla_1\log r)^2\Big),
\end{align*}
where \eqref{3-11} is used in the inequality.
Since $\alpha\geq \beta+1$ and
$$-2\alpha\beta(\nabla_1\log r)(\nabla_1\log F)\leq \beta(\beta+1)(\nabla_1\log F)^2+\frac{\alpha^2\beta}{\beta+1}(\nabla_1 \log r)^2,$$
hence,
\begin{align}\label{3-13}
\partial_t \rho
&\leq \beta\Phi F^{-1}\dot F^{kl}\nabla^2_{kl}\rho+(\beta-1)\Phi +b^{11}-\alpha\Phi(b^{11})^2\frac{1}{r}r_{11}\nn\\
&\quad-\frac{\alpha(\alpha-\beta-1)}{\beta+1}(b^{11})^2\Phi (\nabla_1\log r)^2\nn\\
&\leq \beta\Phi F^{-1}\dot F^{kl}\nabla^2_{kl}\rho+(\beta-1)\Phi +b^{11}-\alpha\Phi(b^{11})^2\frac{1}{r}r_{11}.
\end{align}
On the other hand, notice that
\begin{align*}
\nabla_1 r=r^{-1}\langle X_1, X\rangle,
\end{align*}
 $X_1$ and $\frac{X}{r}$ are unit vectors, we have by Lemma \ref{lem3-5} that
$$g_{11}-(\nabla_1 r)^2=1-\left\langle X_1, \frac{X}{r}\right \rangle^2\geq \delta_0$$
for some constant $\delta_0>0$.
Therefore
\begin{align}\label{3-14}
\nabla^2_{11}r&=(r^{-1}\langle X_1, X\rangle)_1\nn\\
&=-\frac{1}{r^2}\frac{\langle X_1, X\rangle\langle X_1, X\rangle}{r}+\frac{1}{r}\Big(\langle X_{11}, X\rangle+\langle X_1, X_1\rangle\Big)\nn\\
&=r^{-1}\left(-uh_{11}+g_{11}-(\nabla_1r)^2\right)\nn\\
&\geq -\frac{uh_{11}}{r}+\frac{\delta_0}{r}.
\end{align}
Submitting \eqref{3-14} into \eqref{3-13}, we obtain
\begin{align*}
\partial_t \rho
&\leq \beta\Phi F^{-1}\dot F^{kl}\nabla^2_{kl}\rho+(\beta-1)\Phi +\left(1+\frac{\alpha u}{r^2}\Phi\right)b^{11}-\alpha\Phi(b^{11})^2\frac{\delta_0}{r^2}.
\end{align*}
Thus
\begin{align}\label{3-15}
\partial_t \log \rho
&\leq \beta\Phi F^{-1}\dot F^{kl}\nabla^2_{kl}\log\rho+\beta\Phi F^{-1}\dot F^{ij}\nabla_i \log \rho\nabla_j\log \rho\nn\\
&\quad+\frac{\beta-1}{\rho}\Phi+\left(1+\frac{\alpha u}{r^2}\Phi\right)-\alpha\Phi b^{11}\frac{\delta_0}{r^2}.
\end{align}
Notice that
\begin{align*}
\nabla^2_{ij} u&=\nabla_j\langle X, h^l_iX_l\rangle\\
&=h_{ij}+\langle X, \nabla h_{ij}\rangle-h_{ik}h^k_j u,
\end{align*}
by the evolution equation for $\nu$ in Lemma \ref{lem2-2}, we obtain
\begin{align*}
\partial_t u&=\langle-\Phi \nu+X, \nu\rangle+\langle X, \nabla \Phi\rangle\\
&=\langle X, \beta\Phi F^{-1}\dot F^{ij}\nabla h_{ij}+\alpha \Phi r^{-1}\nabla r\rangle-\Phi+u\\
&=\beta\Phi F^{-1}\dot F^{ij} (\nabla^2_{ij}u-h_{ij}+h_{ik}h^k_j u)+\alpha\Phi g^{ij}\langle X_i, \frac{X}{|X|}\rangle\langle X_j, \frac{X}{|X|}\rangle-\Phi+u\\
&\geq \beta\Phi F^{-1}\dot F^{ij}\nabla^2_{ij}u-(\beta+1)\Phi+u,
\end{align*}
which implies
\begin{align}\label{3-16}
\partial_t \log u&\geq \beta\Phi F^{-1}\dot F^{ij}\nabla^2_{ij}\log u+\beta\Phi F^{-1}\dot F^{ij}\nabla_i \log u\nabla_j\log u\nn\\
&\quad -\frac{1}{u}(\beta+1)\Phi+1.
\end{align}
Therefore, at point $(x_0, t_0)$, by \eqref{3-15} and \eqref{3-16} we have
\begin{align*}
0\leq \partial_t w&=\partial_t \log \rho-\partial_t \log u\nn\\
&\leq  \beta\Phi F^{-1}\dot F^{ij}\nabla^2_{ij}(\log \rho-\log u)+\beta\Phi F^{-1}\dot F^{ii}\left((\nabla_i \log \rho)^2-(\nabla_i \log u)^2\right)\nn\\
&\quad +\frac{\beta-1}{\rho}\Phi+\left(1+\frac{\alpha u}{r^2}\Phi\right)-\alpha\Phi b^{11}\frac{\delta_0}{r^2}+\frac{1}{u}(\beta+1)\Phi-1\nn\\
&\leq \frac{\beta-1}{b^{11}}\Phi-\alpha\Phi b^{11}\frac{\delta_0}{r^2}+C,
\end{align*}
where Lemma \ref{lem3-1} and Lemma \ref{lem3-4} are used in the last inequality. By above inequality, we have $b^{11}$ is bounded .

\textbf{Case II}: $\Phi=u^\alpha F^\beta$. By Lemma \ref{lem2-2} and
\begin{align*}
\nabla^2_{ij} u=h_{ij}+\langle X, \nabla h_{ij}\rangle-h_{ik}h^k_j u,
\end{align*}
we have similarly
\begin{align*}
\partial_t \rho
&\leq \beta\Phi F^{-1}\dot F^{kl}\nabla^2_{kl}\rho+(\beta-1)\Phi +b^{11}-\alpha\Phi(b^{11})^2\frac{1}{u}u_{11}\\
&=\beta\Phi F^{-1}\dot F^{kl}\nabla^2_{kl}\rho+(\beta-1)\Phi +(1-\alpha\frac{\Phi}{u})b^{11}+\alpha\Phi+\alpha\frac{\Phi}{u}\langle X, \nabla \rho\rangle,
\end{align*}
which implies
\begin{align}\label{3-13b}
\partial_t \log \rho
&\leq \beta\Phi F^{-1}\dot F^{kl}\nabla^2_{kl}\log\rho+\beta\Phi F^{-1}\dot F^{ij}\nabla_i \log \rho\nabla_j\log \rho\nn\\
&\quad+\frac{\beta-1}{\rho}\Phi+\left(1-\frac{\alpha \Phi}{u}\right)+\frac{\alpha\Phi}{ b^{11}}+\alpha\frac{\Phi}{u}\langle X, \nabla\log \rho\rangle.
\end{align}
Since
\begin{align*}
\partial_t u&=\langle-\Phi \nu+X, \nu\rangle+\langle X, \nabla \Phi\rangle\\
&=\langle X, \beta\Phi F^{-1}\dot F^{ij}\nabla h_{ij}+\alpha \Phi u^{-1}\nabla u\rangle-\Phi+u\\
&=\beta\Phi F^{-1}\dot F^{ij} (\nabla^2_{ij}u-h_{ij}+h_{ik}h^k_j u)+\alpha\Phi \langle X, \nabla \log u\rangle-\Phi+u\\
&\geq \beta\Phi F^{-1}\dot F^{ij}\nabla^2_{ij}u-(\beta+1)\Phi+\alpha\Phi \langle X, \nabla \log u\rangle+u,
\end{align*}
thus,
\begin{align}\label{3-16b}
\partial_t \log u&\geq \beta\Phi F^{-1}\dot F^{ij}\nabla^2_{ij}\log u+\beta\Phi F^{-1}\dot F^{ij}\nabla_i \log u\nabla_j\log u\nn\\
&\quad -\frac{1}{u}(\beta+1)\Phi+1+\alpha\frac{\Phi}{u}\langle X, \nabla \log u\rangle.
\end{align}
It follows from \eqref{3-13b} and \eqref{3-16b} that, at point $(x_0, t_0)$,
\begin{align*}
0\leq \partial_t w&=\partial_t \log \rho-\partial_t \log u\nn\\
&\leq  \beta\Phi F^{-1}\dot F^{ij}\nabla^2_{ij}(\log \rho-\log u)+\beta\Phi F^{-1}\dot F^{ii}\left((\nabla_i \log \rho)^2-(\nabla_i \log u)^2\right)\nn\\
&\quad +\frac{\beta-1}{\rho}\Phi+\left(1-\frac{\alpha \Phi}{u}\right)+\frac{\alpha\Phi}{ b^{11}}+\alpha\frac{\Phi}{u}\langle X, \nabla\log \rho-\nabla \log u\rangle-1\nn\\
&\quad+\frac{1}{u}(\beta+1)\Phi\nn\\
&\leq \frac{\beta-1+\alpha}{b^{11}}\Phi+(\beta+1-\alpha)\frac{\Phi}{u}.
\end{align*}
Notice that $\alpha> \beta+1$, we have $b^{11}$ is bounded.
From this the assertion follows.
\end{proof}

\begin{rem}\label{rem3.1}
For special inverse concave curvature function $F=K^{\frac{s}{n}}F_1^{1-s}(s\in(0, 1])$, where $K$ is the Gauss curvature and $F_1$ satisfies Condition \ref{con-1}, we can also obtain that Lemma \ref{lem3-6} holds when $\alpha=\beta+1$ for the case $\Phi=u^\alpha F^\beta$. The reason are as follows:

Let us define auxiliary function
 $$w(x, t)=\log \rho(x, t)-\varepsilon\log u(x, t),$$
where $\rho(x, t)=\frac{b^{11}}{g^{11}}$ and $\varepsilon$ is a positive constant to be chosen later.
By $\nabla_i \log u=\frac{\lambda_i}{u}$, \eqref{3-13b} and \eqref{3-16b} we can get that, at point $(x_0, t_0)$,
\begin{align}\label{3-21}
0\leq\partial_t w&=\partial_t \log \rho-\varepsilon\partial_t \log u\nn\\
&\leq  \beta\Phi F^{-1}\dot F^{ij}\nabla^2_{ij}(\log \rho-\varepsilon\log u)+\beta\Phi F^{-1}\dot F^{ii}\left((\nabla_i \log \rho)^2-\varepsilon(\nabla_i \log u)^2\right)\nn\\
&\quad +\frac{2\beta}{b^{11}}\Phi+(1-\varepsilon)\left(1-\frac{\beta+1 }{u}\Phi\right)+\alpha\frac{\Phi}{u}\langle X, \nabla\log \rho-\varepsilon\nabla \log u\rangle\nn\\
&\leq \frac{2\beta\Phi}{b^{11}}+(1-\varepsilon)\left(1-(\beta+1)(uF)^\beta-\varepsilon\beta\Phi F^{-1}\dot F^i \frac{\lambda_i^2}{u^2}\right).
\end{align}
It follows from Corollary \ref{cor-1} that $K$ and $F_1$ are bounded from above. If $b^{11}$ is not bounded from above (i.e., $\lambda_1$ is small enough), then by Corollary \ref{cor-1} again we have
$$\dot F^i\lambda_i^2\leq \lambda_n \dot F^i\lambda_i\leq CF=CK^{\frac{s}{n}}F_1^{1-s}\rightarrow 0.$$
Combining Lemma \ref{lem3-1} with Lemma \ref{lem3-4} gives
$$\frac{2\beta\Phi}{b^{11}}-(1-\varepsilon)\left((\beta+1)(uF)^\beta+\varepsilon\beta\Phi F^{-1}\dot F^i \frac{\lambda_i^2}{u^2}\right)\rightarrow 0.$$
Choosing $\varepsilon > 1$, then inequality \eqref{3-21} becomes
$$0\leq \partial_t w<0,$$
it is a contradiction. Hence $b^{11}$ is bounded, that is, the principal curvatures $\lambda_i$ are bounded from below.


\end{rem}

Last we will prove the higher regularity for solutions of  normalized flow \eqref{1-2} or \eqref{1-2b}. As we all know, the radial
graphical representation of flow is often used to derive the key $C^{2,\alpha}$ estimates. But in this paper, we assumed that
$F$ is inverse concave, hence we apply the Gauss map parametrization of flow which has
been used widely for convex hypersurfaces \cite{AMZ,AW,LL1,W} and write flow as a parabolic
equation of support function which is concave with respect to its arguments.
\begin{thm}\label{thm2}
Let $M_0$ be a smooth, closed, uniformly convex hypersurface in $\mathbb R^{n+1}$, which encloses the origin. Suppose $\alpha\geq \beta+1$ when $\Phi=r^\alpha F^\beta$ or $\alpha>\beta+1$ when $\Phi=u^\alpha F^\beta$. Then  the normalized flow \eqref{1-2} or \eqref{1-2b} has a unique smooth, closed and uniformly convex solution $M_t$ for all time $t\geq 0$.

Moreover, for any $k\geq 0$,
the radial function of $M_t$ satisfies the following a priori estimates
\begin{align*}
||D^k r(x, t)||\leq C, \quad\quad (x, t)\in\mathbb S^n\times[0, \infty),
\end{align*}
where $C>0$ depends only on $n, k, \alpha, \beta$ and $M_0$.
\end{thm}
\begin{proof}
\textbf{Case I}: $\Phi=r^\alpha F^\beta$. It follows from Corollary \ref{cor-1}, Lemma \ref{lem3-6} and equation \eqref{2-6} that
$D^2 u$ are uniformly bounded. Thus the equation
\begin{align*}
\partial_t u&=-(u^2+|Du|^2)^{\frac{\alpha}{2}}F^{-\beta}_*(\tau_{ij})+u\nn\\
&\triangleq G(D^2u, Du, u)
\end{align*}
is uniformly parabolic. Straightforward computations give
\begin{align*}
\dot G^{ij}&=\frac{\partial G}{\partial(D^2_{ij}u)}=\beta (u^2+|Du|^2)^{\frac{\alpha}{2}}F^{-\beta-1}_*\dot F_*^{pq}\frac{\partial\tau_{pq}}{\partial(D^2_{ij}u)}\\
&=\beta(u^2+|Du|^2)^{\frac{\alpha}{2}}F^{-\beta-1}_*\dot F_*^{ij},
\end{align*}
and
\begin{align}\label{3-18}
  \ddot G^{ij,kl}
  &=-(\beta+1)\beta(u^2+|Du|^2)^{\frac{\alpha}{2}}F^{-\beta-2}_*\dot F_*^{ij}\dot F_*^{kl}+\beta(u^2+|Du|^2)^{\frac{\alpha}{2}}F^{-\beta-1}_*\ddot F_*^{ij,kl}.
\end{align}
\textbf{Case II}: $\Phi=u^\alpha F^\beta$. Similarly, we can obtain that the
equation
\begin{align*}
\partial_t u&=-u^{\alpha}F^{-\beta}_*(\tau_{ij})+u\triangleq G(D^2u, Du, u)
\end{align*}
is uniformly parabolic, and
\begin{align}\label{3-18b}
  \ddot G^{ij,kl}
  &=-(\beta+1)\beta u^{\alpha}F^{-\beta-2}\dot F_*^{ij}\dot F_*^{kl}+\beta u^{\alpha}F^{-\beta-1}\ddot F_*^{ij,kl}.
\end{align}

By the concavity of $F_*$, from \eqref{3-18} or \eqref{3-18b} we know operator $G$ is concave with respect to $ D^2u$. From the uniform $C^2$ estimates on $u$ in space-time, we can apply the H\"older estimates of \cite{KS} to obtain the $C^{2,\alpha}$ estimate on $u$ and $C^\alpha$ estimate on $\partial_t u$ in space-time. By standard parabolic theory, the bounds on all higher derivatives of $u$ can be established. Therefore, estimates for higher order derivatives of radial function $r$ are established by \eqref{3-2}. Hence the long time existence and smoothness of solutions for the normalized flow \eqref{1-2} or \eqref{1-2b} are obtained. The
uniqueness of smooth solutions also follows from the parabolic theory. Thus this theorem holds.
\end{proof}
\begin{rem}\label{rem3.2}
For special inverse concave curvature function $F=K^{\frac{s}{n}}F_1^{1-s}(s\in(0, 1])$,  by Corollary \ref{cor-1}, Remark \ref{rem3.1} and equation \eqref{2-6}, we can also obtain that Theorem \ref{thm2} holds when $\alpha=\beta+1$ for the case $\Phi=u^\alpha F^\beta$.
\end{rem}
\section{Proof of Theorem \ref{thm1} and Theorem \ref{thm4}}
Previous section shows the solutions $M_t$ to the normalized flow \eqref{1-2} or \eqref{1-2b} exists for all time $t>0$ and remains smooth and uniformly convex. In this section, we will prove the asymptotical convergence to a sphere of $M_t$. We begin by showing the following lemma.
\begin{lem}\label{lem4.1}
Let $X(\cdot, t)$ be a smooth uniformly convex solution to \eqref{1-2} or \eqref{1-2b}. Suppose $\alpha\geq \beta+1$ when $\Phi=r^\alpha F^\beta$ or $\alpha>\beta+1$ when $\Phi=u^\alpha F^\beta$, then there exist positive constants $C$ and  $\gamma$, depending only on $n, \alpha, \beta$ and $M_0$, such that
\begin{align}\label{4-1}
\max_{\mathbb S^n}\frac{|Dr(\cdot, t)|}{r(\cdot, t)}\leq Ce^{-\gamma t}, \quad\quad t>0.
\end{align}
\end{lem}
\begin{proof}
Denote $w=\log r$. By \eqref{2-3}, we have
\begin{align*}
g_{ij}&=e^{2w}(\delta_{ij}+w_iw_j),\\
g^{ij}&=e^{-2w}\left(\delta^{ij}-\frac{w^iw^j}{1+|Dw|^2}\right),\\
h_{ij}&=e^w(1+|Dw|^2)^{-\frac{1}{2}}(\delta_{ij}+w_iw_j-w_{ij})\\
h^i_j&=e^{-w}(1+|Dw|^2)^{-\frac{1}{2}} \tilde g^{ik}a_{kj},
\end{align*}
where $\tilde g^{ik}=\delta^{ik}-\frac{w^iw^k}{1+|Dw|^2}$ and $a_{ij}=\delta_{ij}+w_iw_j-w_{ij}$.

\textbf{Case I}: $\Phi=r^\alpha F^\beta$.
From \eqref{2-4} it follows that
\begin{align}\label{4-2}
\partial_ t w=-(1+|D w|^2)^\frac{1-\beta}{2}e^{(\alpha-\beta-1)w}F^\beta(\tilde g^{ik}a_{kj})+1.
\end{align}
Define auxiliary function
\begin{align*}
Q=\frac{1}{2}|D w|^2.
\end{align*}
At the point where G attains its spatial maximum, we
have
\begin{align}\label{4-3}
0=D_i Q=\sum_k w_kw_{ki},
\end{align}
and
\begin{align}\label{4-4}
0\geq D^2_{ij}Q=\sum_{k}w_kw_{kij}+\sum_kw_{ki}w_{kj}.
\end{align}
By differentiating \eqref{4-2} and combining \eqref{4-3}, we obtain, at the point where $Q$ achieves its spatial maximum
\begin{align}\label{4-5}
\partial_t Q&=w_kw_{kt}=w_k\left (-(1+|D w|^2)^\frac{1-\beta}{2}e^{(\alpha-\beta-1)w}F^\beta(\tilde g^{ik}a_{kj})+1\right)_k\nn\\
&=-(1+|D w|^2)^\frac{1-\beta}{2}e^{(\alpha-\beta-1)w}\left ((\alpha-\beta-1)|Dw|^2F^\beta+\beta F^{\beta-1}\dot F^j_iD_k(\tilde g^{is}a_{sj})w_k\right).
\end{align}
By \eqref{4-3}, straightforward computation gives
\begin{align*}
w_kD_k{\tilde g^{is}}=-w_k\left(\frac{w^iw^s}{1+|Dw|^2}\right)_k=0,
\end{align*}
and
\begin{align*}
w_kD_ka_{sj}=w_k(w_jw_{sk}+w_sw_{jk}-w_{sjk})=-w_kw_{sjk}.
\end{align*}
Notice the Ricci identity
$$w_{ijk}=w_{ikj}+\delta_{ki}w_j-\delta_{ij}w_{ki},$$
we have that
\begin{align*}
-\tilde g^{is}w_kD_ka_{sj}&=\tilde g^{is}(w_{skj}+\delta_{ks}w_j-\delta_{sj}w_k)w_k\\
&\leq \tilde g^{is}(-w_{sk}w_{kj}+w_sw_j-\delta_{sj}|Dw|^2)\nn\\
&\leq \tilde g^{is}(w_sw_j-\delta_{sj}|Dw|^2)\nn\\
&=w^iw_j-\delta^i_j|Dw|^2,
\end{align*}
where \eqref{4-4} is used in the first inequality.
Hence \eqref{4-5} can be rewrite as, by $\alpha\geq \beta+1$,
\begin{align*}
\partial_t Q
&=(1+|D w|^2)^\frac{1-\beta}{2}e^{(\alpha-\beta-1)w}\beta F^{\beta-1}\dot F^{ij}(w_iw_j-\delta_{ij}|Dw|^2)\nn\\
&\leq (1+|D w|^2)^\frac{1-\beta}{2}e^{(\alpha-\beta-1)w}\beta F^{\beta-1}(\max_i \dot F^{ii}-\sum_i\dot F^{ii})|Dw|^2.
\end{align*}
By Corollary \ref{cor-1}, we have
$$\max_i \dot F^{ij}-\sum_i\dot F^{ii}\leq -C,$$
where $C$ depends on $n, \alpha, \beta$ and $M_0$.
Combining Lemma \ref{lem3-1}, Lemma \ref{lem3-2}, Corollary \ref{cor-1} with Lemma \ref{lem3-6} gives
\begin{align}\label{4-6}
  \partial_t Q\leq -\gamma_1 Q,
\end{align}
for some positive constant $\gamma_1$ depending on $n, \alpha, \beta$ and $M_0$.

\textbf{Case II}: $\Phi=u^\alpha F^\beta$. From equation \eqref{2-4b} and  $$u=\frac{r}{\sqrt{1+r^{-2}|Dr|^2}}=\frac{e^w}{\sqrt{1+|Dw|^2}},$$ we can obtain that, by similar arguments,
\begin{align*}
\partial_t Q
&\leq (1+|D w|^2)^\frac{1-\beta-\alpha}{2}e^{(\alpha-\beta-1)w}\beta F^{\beta-1}(\max_i \dot F^{ii}-\sum_i\dot F^{ii})|Dw|^2
\end{align*}
and
\begin{align}\label{4-6b}
  \partial_t Q\leq -\gamma_2 Q,
\end{align}
for some positive constant $\gamma_2$ depending on $n, \alpha, \beta$ and $M_0$.

From inequality \eqref{4-6} or \eqref{4-6b}, the inequality \eqref{4-1} follows.
\end{proof}

\begin{rem}\label{rem4.1}
For special inverse concave curvature function $F=K^{\frac{s}{n}}F_1^{1-s} (s\in (0, 1])$, by Lemma \ref{lem3-1}, Lemma \ref{lem3-2}, Corollary \ref{cor-1} and Remark \ref{rem3.1} we can also get
\begin{align*}
  \partial_t Q\leq -\gamma_3 Q,
\end{align*}
for some positive constant $\gamma_3$ depending on $n, \alpha, \beta$ and $M_0$.
Therefore, for the case $\Phi=u^\alpha F^\beta$ and $\alpha=\beta+1$, we also have
\begin{align}\label{4-10}
\max_{\mathbb S^n}\frac{|Dr(\cdot, t)|}{r(\cdot, t)}\leq Ce^{-\gamma t}, \quad\quad t>0,
\end{align}
 for some constants $C$ and $\gamma$.
\end{rem}

Next we will prove the Theorem \ref{thm1}.
{\renewcommand{\proofname}{ \textbf{Proof of Theorem \ref{thm1}}}

\begin{proof}
 \quad
\textbf{Case (I): $\Phi=r^\alpha F^\beta$.}

When $\alpha>\beta+1$, let $X(\cdot, t)$ be the solution of \eqref{1-2}. Without loss of generality,
we may assume the radial function $r(\cdot, 0)$ of initial hypersurface $M_0$ satisfies
$$a:=\min_{\mathbb S^n} r(\cdot, 0)\leq 1\leq \max_{\mathbb S^n} r(\cdot, 0)=:b.$$
Let $q=\beta+1-\alpha$, and define function
\begin{align*}
r_1(t)&=\left(1-(1-a^q)e^{qt}\right)^{\frac{1}{q}},\\
r_2(t)&=\left(1-(1-b^q)e^{qt}\right)^{\frac{1}{q}}.
\end{align*}
It is easy to check that 
the spheres of radii $r_1$ and $r_2$ are solutions of \eqref{1-2}. By the comparison principle, $r_1(t)\leq r(\cdot, t)\leq r_2(t)$. Hence
$$(b^q-1)e^{qt}\leq r^q-1\leq (a^q-1)e^{qt}.$$
Thus $r$ converges to $1$ exponentially.

In order to prove radial function $r$ converges to $1$ exponentially in the $C^k$ norms, we use the following interpolation inequality \cite{Ham}:
\begin{align}\label{4-7}
\int_{\mathbb S^n}|D^k S|^2\leq C_{m, n}\left(\int_{\mathbb S^n}|D^m S|^2\right)^{\frac{k}{m}}\left(\int_{\mathbb S^n}|S|^2\right)^{1-\frac{k}{m}},
\end{align}
where $S$ is any smooth tensor field on $\mathbb S^n$, and $k, m$ are any integers such that $0\leq k\leq m$. Choosing $S=r-1$ in \eqref{4-7} and using the estimates for higher order derivatives of $r$ in Theorem \ref{thm2}, we conclude
\begin{align}\label{4-8}
\int_{\mathbb S^n}|D^k r|^2\leq C_{k, \delta}e^{-\delta t},
\end{align}
for any $\delta\in(0, \tilde\delta)$ and any positive integer $k$, where $\tilde \delta>0$ is a constant depending only
on $q$. By the Sobolev embedding theorem on $\mathbb S^n$ (\cite{Aub}), we have
\begin{align}\label{4-9}
||r-1||_{C^l(\mathbb S^n)}\leq C_{k,l}\left(\int_{\mathbb S^n}|D^k r|^2+\int_{\mathbb S^n}|r-1|^2\right)^{\frac{1}{2}}
\end{align}
for any $k>l+\frac{n}{2}$. Hence, by \eqref{4-8} and \eqref{4-9}, we have $||r-1||_{C^l(\mathbb S^n)}\rightarrow 0$ exponentially as $t\rightarrow \infty$ for all integers $l\geq 1$.  As a conclusion, $M_t$ converges exponentially to the unit sphere centered at the origin in the $C^\infty$ topology.


When $\alpha=\beta+1$, it follows from \eqref{4-1} that $|Dr|\rightarrow 0$ exponentially as $t\rightarrow \infty$, that is, $r$
converges exponentially to a constant as $t\rightarrow \infty$.
Hence by the interpolation and the a priori estimates for $r$ again,  we can deduce that $r$
converges exponentially to a constant in the $C^\infty$ topology as $t\rightarrow\infty$. Therefore, $M_t$ converges exponentially to a sphere centered at the origin in the $C^\infty$ topology.

\textbf{Case (II): $\Phi=u^\alpha F^\beta$.} Since $\alpha>\beta+1$, we can assume that
 the support function $u(\cdot, 0)$ of initial hypersurface $M_0$ satisfies
$$a:=\min_{\mathbb S^n} u(\cdot, 0)\leq 1\leq \max_{\mathbb S^n} u(\cdot, 0)=:b.$$
Let $q=\beta+1-\alpha$, and define function
\begin{align*}
u_1(t)&=\left(1-(1-a^q)e^{qt}\right)^{\frac{1}{q}},\\
u_2(t)&=\left(1-(1-b^q)e^{qt}\right)^{\frac{1}{q}}.
\end{align*}
By similar arguments as in the case $\Phi=r^\alpha F^\beta$ and $\alpha>\beta+1$, we can get that
$M_t$ converges exponentially to the unit sphere centered at the origin in the $C^\infty$ topology.

\end{proof}}

Last, we prove the Theorem \ref{thm4}.
{\renewcommand{\proofname}{ \textbf{Proof of Theorem \ref{thm4}}}
\begin{proof}
\quad For special inverse concave curvature function $F=K^{\frac{s}{n}}F_1^{1-s} (s\in (0, 1])$, by Remark \ref{rem4.1} we have $|Dr|\rightarrow 0$ exponentially as $t\rightarrow \infty$, that is, $r$
converges exponentially to a constant as $t\rightarrow \infty$.
Hence by the interpolation and the a priori estimates for $r$ (Remark \ref{rem3.2}) again,  we can deduce that $r$
converges exponentially to a constant in the $C^\infty$ topology as $t\rightarrow\infty$. Therefore, the rescaled hypersurfaces $M_t$ converges exponentially to a sphere centered at the origin in the $C^\infty$ topology.
\end{proof}

}
\section{Proof of Theorem  \ref{thm3}}

In this section, we will show that if $\alpha<\beta+1$ then there exists  a smooth closed uniformly convex hypersurfaces to the flow \eqref{1-1} such that the ratio of radii is unbounded, that is,
\begin{align}\label{5-1}
\mathcal R(X(x, t))=\frac{\max_{\mathbb S^n} r(\cdot, t)}{\min_{\mathbb S^n} r(\cdot, t)}\rightarrow \infty\quad{\text as}\quad t\rightarrow T.
\end{align}

We begin by recalling the following definition of sub-solution.
\begin{defn}
A time-dependent family of convex hypersurfaces $Y(\cdot, t)$ is a sub-solution to \eqref{1-1} if
\begin{align*}
    \begin{cases}
    \frac{\partial}{\partial t}Y(\cdot, t)\geq-r^\alpha F^{\beta}(\lambda(Y)) , \\Y(\cdot,0)=Y_0(\cdot),\end{cases}
\end{align*}
where  $r(\cdot, t)$ is the radial function of $Y(\cdot, t)$ and $\lambda(Y)=(\lambda_1, \ldots, \lambda_n)$ with $\lambda_i$ are the principal curvatures of the hypersurface $Y(\cdot, t)$.
\end{defn}

Furthermore, we will need the following comparison principle.
\begin{lem}\label{lem5-2}
Let $X(\cdot, t)$ be a solution to \eqref{1-1} and $Y(\cdot, t)$ a sub-solution. Suppose $X(\cdot, t)$ is contained in the interior $Y(\cdot, t)$. Then $X(\cdot, t)$ is contained in the interior $Y(\cdot, t)$ for all $t>0$, as long as the solution exists.
\end{lem}

To prove Theorem \ref{thm3}, by the arguments that appeared in \cite{LSW1,LSW}, it suffices to construct a sub-solution $Y(\cdot, t)$ such that $\min_{\mathbb S^n} r(\cdot, t)\rightarrow 0$ in finite time while $\max_{\mathbb S^n} r(\cdot, t)$ remains positive. By a translation of time, we will construct a sub-solution $Y(\cdot, t)$ for $t\in(-1, 0)$ such that \eqref{5-1} holds as $t\nearrow 0$.

\begin{lem}\label{lem5-1}
There is a sub-solution $Y(\cdot, t)\left(t\in(-1, 0)\right)$, to
\begin{align}\label{5-3}
    \begin{cases}
    \frac{\partial}{\partial t}X(\cdot, t)=-ar^\alpha F^{\beta}(\lambda(X)) , \\X(\cdot,0)=X_0(\cdot),\end{cases}
\end{align}
for a sufficiently large constant $a>0$, such that its radial function $r(\cdot, t)$ satisfies $\mathop{\min}\limits_{\mathbb S^n} r(\cdot, t)\rightarrow 0$ but $\mathop{\max}\limits_{\mathbb S^n} r(\cdot, t)$
remains positive, as $t\nearrow 0$.
\end{lem}

\begin{proof}
 It's well known that, if $\widehat M_t=:Y(\mathbb S^n, t)$ is a sub-solution to \eqref{5-3} for some $\alpha$,  when we replace $a$ by another constant
$$a\sup\{|p|^{\alpha-\alpha'}: p\in \widehat M_t, t\in(-1, 0)\},$$
then it is also a sub-solution
to \eqref{5-3} for $\alpha'<\alpha$. Thus we only need to prove Lemma \ref{lem5-1} when
$\gamma=\beta+1-\alpha>0$ is very small.

Let $\widehat M_t$ be the graph of the function
\begin{align}\label{5-4}
    \varphi(\rho, t)=
    \begin{cases}
    -|t|^\theta+|t|^{(\sigma-1)\theta}\rho^2, \quad\quad\quad\quad\quad\quad\quad~ \rho<|t|^\theta,\\
    -|t|^{\theta}-\frac{1-\sigma}{1+\sigma}|t|^{(1+\sigma)\theta}
    +\frac{2}{1+\sigma}\rho^{1+\sigma},\quad |t|^\theta\leq \rho\leq 1,\end{cases}
\end{align}
where $\sigma=\frac{\gamma\theta-1}{\beta\theta}\in(0, 1)$, $\theta>\frac{1}{\gamma}$ is a constant, and  $\rho=|x|$ with $x\in \mathbb R^n$. It is easy to verify that $\varphi\in C^{1,1}(B_1(0))$ and $\varphi$ is strictly convex.

Since $\widehat M_t={\text graph}~\varphi$, then the induced metric and its inverse can be expressed as
\begin{align*}
g_{ij}=\delta_{ij}\varphi_i\varphi_j\quad\quad{\text and}\quad\quad g^{ij}=\delta^{ij}-\frac{\varphi^i\varphi^j}{1+|\bar\nabla \varphi|^2},
\end{align*}
where $\varphi_i$ is the partial derivative of $\varphi$. After a standard computation, the
second fundamental form can be expressed as
$$h_{ij}=\frac{\varphi_{ij}}{\sqrt{1+|\bar\nabla \varphi|^2}},$$
which implies the matrix of the Weingarten map is
\begin{align*}
h^i_{j}=\frac{\varphi_{jk}}{\sqrt{1+|\bar\nabla \varphi|^2}}\left(\delta^{ik}-\frac{\varphi^i\varphi^k}{1+|\bar\nabla \varphi|^2}\right).
\end{align*}
We can deduce that by direct computation
\begin{align*}
h^i_{j}=\begin{cases}
\frac{2|t|^{(\sigma-1)\theta}}{\sqrt{1+4|t|^{2(\sigma-1)\theta}\rho^2}}
\left(\delta^{i}_j-\frac{4|t|^{2(\sigma-1)\theta}x^ix^j}
{1+4|t|^{2(\sigma-1)\theta}\rho^2}\right),\quad\quad\quad 0\leq \rho\leq |t|^\theta,\\
\frac{2\rho^{\sigma-1}}{\sqrt{1+4\rho^{2\sigma}}}\left(\delta^{i}_j-\Big(\frac{4\sigma\rho^{2(\sigma-1)}}
{1+4\rho^{2\sigma}}+\frac{1-\sigma}{\rho^2}\Big)x^ix_j\right), \quad\quad |t|^\theta\leq \rho\leq 1.
\end{cases}
\end{align*}
Thus the principal curvatures of $\widehat M_t$ are
\begin{align*}
\lambda_1=\frac{2|t|^{(\sigma-1)\theta}}{(1+4|t|^{2(\sigma-1)\theta}\rho^2)^{\frac{3}{2}}}, \quad\lambda_2=\cdots=\lambda_n=\frac{2|t|^{(\sigma-1)\theta}}{\sqrt{1+4|t|^{2(\sigma-1)\theta}\rho^2}},
\end{align*}
and
\begin{align*}
\lambda_1=\frac{2\sigma\rho^{\sigma-1}}
{(1+4\rho^{2\sigma})^{\frac{3}{2}}}, \quad \lambda_2=\cdots=\lambda_n=\frac{2\rho^{\sigma-1}}{\sqrt{1+4\rho^{2\sigma}}},
\end{align*}
respectively  when $0\leq \rho\leq |t|^\theta$ and $|t|^\theta\leq \rho\leq 1$.
It follows from Condition \ref{con-1} (iii) that
$$r^{\alpha}F^{\beta}\geq |t|^{\alpha\theta}F^\beta\geq |t|^{\alpha\theta}\left(\frac{2|t|^{(\sigma-1)\theta}}
{(1+4|t|^{2(\sigma-1)\theta}\rho^2)^{\frac{3}{2}}}\right)^\beta\geq C |t|^{\alpha\theta}|t|^{(\sigma-1)\theta\beta}=C|t|^{\theta-1},
$$
and
\begin{align*}
r^{\alpha}F^{\beta}&\geq \rho^{\alpha}F^\beta\geq \rho^{\alpha}\left(\frac{2\sigma\rho^{\sigma-1}}
{(1+4\rho^{2\sigma})^{\frac{3}{2}}}\right)^\beta
\geq C \rho^{\alpha}\rho^{(\sigma-1)\beta}
\geq C|t|^{\theta-1},
\end{align*}
respectively  when $0\leq \rho\leq |t|^\theta$ and $|t|^\theta\leq \rho\leq 1$.

On the other hand, by the definition of $\varphi$, it is easy to verify
$$\left|\frac{\partial}{\partial t}Y(\cdot, t)\right|\leq 2\theta|t|^{\theta-1}.$$
Thus the graph of $\varphi(\cdot, t)$ is a sub-solution to \eqref{5-3}, provided $a$ is sufficiently large.

Next we extend the graph of $\varphi$ to a closed convex hypersurface
$\widehat M_t$, such that it is $C^{1, 1}$ smooth, uniformly convex, rotationally symmetric and depends smoothly on $t$.
Moreover we may assume that the ball $B_1(z)$ is contained in the interior of
$\widehat M_t$, for
all $t\in(-1, 0)$, where $z=(0, \ldots, 0, 10)$ is a point on the $x_{n+1}$-axis. Then $\widehat M_t$ is a
sub-solution to \eqref{5-3}, for sufficiently large $a$.
\end{proof}

Based on this lemma, we can prove Theorem \ref{thm3}  by similar arguments that appeared in \cite{LSW1,LSW}. For the convenience of readers, we give the proof in detail.

{\renewcommand{\proofname}{ \textbf{Proof of Theorem \ref{thm3}}}

\begin{proof}
\quad Given $\tau\in(-1, 0)$, let $M_0$ be
a smooth, closed, uniformly convex hypersurface inside $\widehat M_\tau$ and enclosing the ball $B_1(z)$. Suppose $M_t$ be the solution to the flow \eqref{5-3} with initial data $M_0$. It follows from Lemma \ref{lem5-2} that, there exists $t_0\in(\tau, 0)$ such that
$M_t$ touches the origin at $t=t_0$. Choose $\tau$  close enough to $0$ such
that $t_0$ is  small enough.

Let $\tilde X(\cdot, t)$ be the solution to
\begin{align}
\frac{\partial X}{\partial t}=-ba\tilde r^\alpha F^\beta\nu,
\end{align}
with initial data $\tilde X(\cdot, \tau)=\partial B_1(z)$, where $b=2^\alpha\sup\{|p|^\alpha: p\in M_t, \tau<t<t_0 \}$, $a$ is a sufficiently large constant in Lemma \ref{lem5-1}, and $\tilde r=|X-z|$ is the distance from $z$ to $X$. Choose $\tau$  small enough such that the ball $B_{1/2}(z)$ is contained in the interior of $\tilde X(\cdot, t)$ for all $t\in(\tau, t_0)$. Then the ball $B_{1/2}(z)$ is contained in the
interior of $M_t$ for all $t\in(\tau, t_0)$ by Lemma \ref{lem5-2}. Hence as $t\nearrow t_0$, we have $\min_{\mathbb S^n} r(\cdot, t)\rightarrow 0$ but
$\max_{\mathbb S^n} r(\cdot, t)>|z|=10$. Hence \eqref{5-1} is proved for $M_t$.

For large constant $a>0$, Theorem \ref{thm3} is obtained when $r^\alpha F^{\beta}$ is replaced by $ar^{\alpha}F^\beta$. Making the rescaling
$\widetilde M_t=a^{-\frac{1}{\gamma}}M_t$, we can prove that
$\widetilde{M_t}$ solves the
flow \eqref{1-1}. From this, Theorem \ref{thm3} is proved.
\end{proof}




\end{document}